\documentclass{article}
\usepackage[utf8]{inputenc}

\usepackage{hyperref}
\usepackage[margin=1.3in]{geometry}
\usepackage{dcpic,pictex}
\usepackage{tikz-cd}  

\usepackage{amsfonts}
\usepackage{amssymb}
\usepackage{mathrsfs}
\usepackage{amsmath,amsthm}

\usepackage{enumerate}
\usepackage{enumitem}

\usepackage{color}
\usepackage{indentfirst}
\usepackage[all]{xy}

\theoremstyle{plain} 
\newtheorem{thm}{Theorem}[section]
\newtheorem*{thm*}{Theorem}
\newtheorem{cor}[thm]{Corollary}
\newtheorem*{cor*}{Corollary}
\newtheorem{lem}[thm]{Lemma}
\newtheorem{prop}[thm]{Proposition}
\numberwithin{equation}{section}

\theoremstyle{definition}
\newtheorem*{defn*}{Definition}
\newtheorem*{rem*}{Remark}
\newtheorem*{exa*}{Example}

\newcommand{\Z}{\mathbb{Z}}
\newcommand{\Q}{\mathbb{Q}}

\title{Degree $4$ cohomological invariants of algebraic tori}

\author{Cyril Demarche 	\and 	Hanqing Long}

\date{\today}

\begin{document}
	
	\maketitle
	
	\bibliographystyle{alpha}

	\begin{abstract}
		In this paper, we determine the motive of the classifying torsor of an algebraic torus.
		As a result, we give an exact sequence describing the degree 4 cohomological invariants of algebraic tori. Using results by Blinstein and Merkurjev, this provides a formula for the degree 4 unramified cohomology group of an algebraic torus, via a flasque resolution.
	\end{abstract}
	
	\section{Introduction}
	
	Let $G$ be an algebraic group over a field $F$.
	An invariant of $G$ \cite{garibaldi2003cohomological} is a natural transformation from the functor of $G$-torsors over fields to a functor $H$ from category of fields to category of abelian groups. 
	We denote this invariant group by $\mathrm{Inv}(G,H)$.
	We mainly consider the cohomological invariants, i.e. the case that $H$ is a Galois cohomological functor $H_\mathrm{\acute{e}t}^{d}(-,\mathbb{Q/Z}(j))$ for $d\ge 0, j\ge 0$, and denote this group by $\mathrm{Inv}^{d}(G,\mathbb{Q/Z}(j))$.
	
	In \cite{m2013}, Blinstein and Merkurjev provide a way to compute $\mathrm{Inv}^{d}(G,\mathbb{Q/Z}(j))$ through the motivic cohomology of a classifying space $BG$ of $G$.
	For an algebraic torus $T$, they provide an exact sequence describing $\mathrm{Inv}^{3}(T,\mathbb{Q/Z}(2))_{norm}$, which is the subgroup of normalized invariants, and they give a similar exact sequence computing the unramified cohomology group $H^3_{nr}(F(T),\Q/\Z(2))$ via a flasque resolution of $T$. In particular, they provide an isomorphism 
	\[\gamma_0 : \bar{H}^4(BT, \Z(2))_{bal} / \mathrm{CH}^2(BT) \xrightarrow{\sim} \mathrm{Inv}^{3}(T,\mathbb{Q/Z}(2))_{norm} \] (see \cite{m2013}, diagram before Lemma 4.2) and the exact sequence recalled below in Theorem \ref{invariants3}. Applying these results, Wei and the second author give explicit formulas in \cite{long20243} that determine completely the third unramified cohomology group for norm $1$ tori associated with abelian field extensions. Understanding unramified cohomology of tori is a longstanding problem related to rational points and rationality questions, and in \cite{ct1995}, Colliot-Th\'el\`ene asks for a formula computing these groups for an arbitrary torus $T$, in terms of the module $\widehat{T}$ of characters of $T$.
	
	Following the approach in \cite{m2013}, we propose to compute the groups $\mathrm{Inv}^{4}(T,\mathbb{Q/Z}(3))_{norm}$ and $H^4_{nr}(F(T),\Q/\Z(3))$ for all tori $T$. In the statement below, a classifying space of $T$ is defined as $BT := U/T$, where $V$ is a generically free representation of $T$ and $U \subset V$ in an open subset on which the action of $T$ is free and such that the codimension of $V \setminus U$ in $V$ is big enough. All cohomology groups are \'etale or \'etale motivic cohomology groups.
	
	\newpage
	\noindent Let $I$ be the kernel of \'etale cycle map $\mathrm{CH}^3(BT) \to \bar{H}^6_\mathrm{\acute{e}t}(BT,\mathbb{Z}(3))$.
	\begin{thm*}
		Let $T$ be a torus over a perfect field $F$. 
		Then there is a natural commutative diagram:
		\[
		\begin{tikzcd}
			& [-2.5em] \left(\widehat{T} \otimes K^M_2(F_{sep})\right)^\Gamma \arrow[d] &[-2.5em]&[-2em]  \mathrm{Inv}^3(T, \mathbb{Q/Z}(2))_{norm}\otimes F^*\arrow[lldd, "\gamma" above]\arrow[ldd, "\cup" above]&[-1.5em]   \\
			& H^2(F,\widehat{T}\otimes  \mathbb{Q/Z}(2)) \arrow[d]& &0\arrow[d]&\\
			A^2(BT, K_3^M) \arrow[r,hook]&\bar{H}_\mathrm{\acute{e}t}^5(BT, \Z(3))\arrow[d]\arrow[r]&\mathrm{Inv}^4(T, \mathbb{Q/Z}(3))_{norm}\arrow[r]&I \arrow[r]\arrow[d]& 0\\
			&\left(S^2(\widehat{T})\otimes F_{sep}^*\right)^\Gamma\arrow[dr]& &\mathrm{CH}^3(BT)_{tors} \arrow[d]&  \\
			&&H^3(F,\widehat{T}\otimes  \mathbb{Q/Z}(2))\arrow[r]&\ker\left(\bar{H}_\mathrm{\acute{e}t}^6(BT,\Z(3)) \to S^3(\widehat{T})^\Gamma\right) \arrow[r] &H^1(F,S^2(\widehat{T})\otimes F^*_{sep})
		\end{tikzcd}
		\]
		where the row, the column and the column that turns to a line are exact, where $\gamma$ is induced by the map $\gamma_0$ above and cup-product.
		This conclusion holds for $\mathrm{char}\ F=0$. 
		If $\mathrm{char}\ F=p>0$, then the conclusion still holds after changing $\mathbb{Q/Z}(2)$ in the Galois cohomology groups of $F$ to $\mathbb{Q/Z}(2)[1/p]$.
	\end{thm*}
	Note that the group $I$ is also the kernel of \'etale cycle map $\mathrm{CH}^3(BT) \to \bar{H}^6_\mathrm{\acute{e}t}(BT,\mathbb{Z}(3))$.
	
	In addition, let $1\longrightarrow T\longrightarrow P\longrightarrow S\longrightarrow 1$ be a flasque resolution of an arbitrary $F$-torus $S$.
	By results of Blinstein and Merkurjev \cite{m2013}, the fourth unramified cohomology group $\bar{H}^4_{nr}(F(S), \mathbb{Q/Z}(3))$ is isomorphic to $\mathrm{Inv}^4(T, \mathbb{Q/Z}(3))_{norm}$, therefore the previous theorem also gives an exact sequence computing this unramified cohomology group of $S$.

	The key point of the proof of the main Theorem is to compute the \'etale motivic cohomology of a classifying $T$-torsor $BT$ as above, using the so-called slice filtration spectral sequence introduced in \cite{huber2006slice}.

	\bigskip

	\section{Preliminaries}\label{sec:pre}
	
	In this section, we introduce notations and lemmas which will be helpful in the next sections.

	\subsection{Algebraic tori}
	
	Let $F$ be a field, $F_{sep}$ be a separable closure of $F$ and $\Gamma := \mathrm{Gal}(F_{sep}/F)$.
	Recall that an algebraic torus of dimension $n$ over $F$ is an algebraic $F$-group $T$ such that $T_{sep} := T \times_F F_{sep}$ is isomorphic to $\mathbb{G}_m^{n}$ as a $F_{sep}$-group. 
	For an algebraic $F$-torus $T$, we write $\widehat{T}$ for the $\Gamma$-module $\mathrm{Hom}(T,\mathbb{G}_m)$ and call it the character lattice of $T$. 
	The contravariant functor $T \mapsto \widehat{T}$ defines an anti-equivalence between the
	category of algebraic tori over $F$ and the category of $\Gamma$-lattices. 

	Let $K/F$ be a finite Galois extension with $G=\mathrm{Gal}(K/F)$. 
	We embed $K$ to $F_{sep}$, so that the absolute Galois group $\Gamma_K := \mathrm{Gal}(F_{sep}/K)$ is a subgroup of $\Gamma$.

	A torus $T$ is called flasque (resp. coflasque) if $H^1(L,\widehat{T}_{sep}^\circ)=0$ (resp. $H^1(L,\widehat{T}_{sep})=0$) for all finite field extensions $L/F$. 
	A flasque resolution of a torus $T$ is an exact sequence of tori 
	$$
	1\longrightarrow Q\longrightarrow P\longrightarrow T\longrightarrow 1
	$$
		If $T$ is a norm one torus $R^{(1)}_{L/F}(\mathbb{G}_{m,L})$, the degree $3$ normalized cohomological invariants group of $T$ is isomorphic to $\mathrm{Br(L/F)}$ (see for instance \cite{m2013}, Example 4.14).

	\subsection{Cohomological invariants}
	Given a linear algebraic group $G$ over a field $F$ and a functor $H$ from the category of fields over $F$ to the category of abelian groups, an invariant of $G$ with values in $H$ (or an $H$-invariant) is a morphism of functor:
	$$
	i: H^1(-,G)\to H
	$$
	where $H^1(-,G)$ is the Galois cohomology functor from the category of fields over $F$ to the category of sets.
	We denote the group of $H$-invariants of $G$ by $\mathrm{Inv}(G,H)$.
	An $H$-invariant is called normalized if it takes the trivial torsor to $0$.
	We denote the subgroup of normalized $H$-invariants of $G$ by $\mathrm{Inv}(G,H)_{norm}$.
	There is a natural isomorphism: 
	$$\mathrm{Inv}(G,H)\cong \mathrm{Inv}(G,H)_{norm}\oplus H(F)$$
	where each element $h$ in $H(F)$ defines an $H$-invariant $i_h$ which takes all $G$-torsor over $K$ to the natural image of $h$ in $H(K)$ for every field extension $K/F$. Such an invariant is called a constant invariant.
	We consider here the cohomological invariants, i.e. the case where $H$ is a Galois cohomological functor $H_\mathrm{\acute{e}t}^{d}(-,\mathbb{Q/Z}(j))$ for $d\ge 0, j\ge 0$, and we denote the corresponding group by $\mathrm{Inv}^{d}(G,\mathbb{Q/Z}(j))$.
	
	A $G$-torsor $E\to X$ is said to be classifying if for every field extension $K/F$ with $K$ infinite, every $G$-torsor over $K$  is isomorphic to the fiber of $E \to X$ at a $K$-point of $X$. Classifying torsors exist for all linear $F$-groups by \cite[2b]{m2013}, \cite[5.3]{garibaldi2003cohomological}.
	In \cite{m2013}, Blinstein and Merkurjev give an approach to compute cohomological invariants of smooth linear algebraic groups and an exact estimate of $\mathrm{Inv}^{3}(T,\mathbb{Q/Z}(2))_{norm}$, when $T$ is a $F$-torus. We recall those results here:
	
	\begin{lem}\cite[Theorem 3.4]{m2013}\label{invariants}
		Let $G$ be a smooth linear algebraic group over a field $F$. We assume
		that $G$ is connected if $F$ is a finite field.
		Let $E \to X$ be a classifying $G$-torsor with $E$ a $G$-rational variety such that $E(F)\neq \emptyset$.
		Then the following homomorphism
		$$\varphi : 
		\bar{H}^0_{\mathrm{Zar}}(X,\mathcal{H}^n(\mathbb{Q/Z}(j))_{\mathrm{bal}}\to \mathrm{Inv}^n(G,\mathbb{Q/Z}(j))_{norm}$$
		is an isomorphism.
	\end{lem}
	Here $\mathcal{H}^n(\mathbb{Q/Z}(j))$ is the Zariski sheaf on $X$ associated to the presheaf $U\mapsto H_\mathrm{\acute{e}t}^{n}(U,\mathbb{Q/Z}(j))$ and $H(X)_{\mathrm{bal}}$ is the subgroup of $H(X)$ consists of the elements $h\in H(X)$ such that $p_1^*(h)= p_2^*(h)$ for the projections $p_i: E^2/G\to X$.

	\begin{lem}\cite[Theorem B]{m2013}\label{invariants3}
		Let $T$ be an algebraic torus over a field $F$. Then there is an exact sequence
		\begin{align*}
			0\longrightarrow \mathrm{CH}^2(BT)_{tors}\longrightarrow H^1(F, T^\circ)\longrightarrow &\mathrm{Inv}^{3}(T,\mathbb{Q/Z}(2))_{norm}\\
			\longrightarrow &S^2(\widehat{T})^\Gamma /Dec \longrightarrow H^2(F, T^\circ).
		\end{align*}
	\end{lem}
	Here $Dec$ denotes the image of $\mathrm{CH}^2(BT)$ in $\mathrm{CH}^2(BT_{sep})\simeq S^2(\widehat{T})$.

	\subsection{Motivic cohomology}
	
	Let $F$ be a perfect field.
	We denote by $DM_{-}^{eff}(F)$ (resp. $DM_{-,\acute{e}t}^{eff}(F)$) the category of complexes of Nisnevich (resp. \'{e}tale) sheaves with transfers with homotopy invariant cohomology sheaves, and $\alpha^*$ the natural functor $DM_{-}^{eff}(F)\to DM_{-,\acute{e}t}^{eff}(F)$.
	Let $DM_{gm}^{eff}(F)$ be the category of effective geometrical motives which was defined by Voevodsky \cite{voevodsky2000triangulated}.
	There is a functor from $DM_{gm}^{eff}(F)$ to $DM_{-}^{eff}(F)$, induced by $\left[ X\right] \mapsto \underline{C}_*(X)$, where $\underline{C}_*(X)$ is the Suslin complex of $X$: this complex will be denoted by $M(X)$.

	In \cite{huber2006slice}, Huber and Kahn construct spectral sequences as follows, which are called slice spectral sequences :
	\begin{align*}
		&E_2^{p,q} = H^{p-q}(c_qM,\mathbb{Z}(n-q))\implies H^{p+q}(M,\mathbb Z(n))\\
		&E_2^{p,q} = H^{p-q}_\mathrm{\acute{e}t}(\alpha^*c_qM,\mathbb{Z}(n-q)_\mathrm{\acute{e}t})\implies H^{p+q}_\mathrm{\acute{e}t}(M,\mathbb Z(n)_\mathrm{\acute{e}t})
	\end{align*}
	where 
	$$
	\mathbb Z(n)_\mathrm{\acute{e}t}=
	\begin{cases}
		\alpha^*\mathbb{Z}(n), \quad n\ge0\\
		\bigoplus_{l\neq \mathrm{char} F}\mathbb{Q}_l/\mathbb{Z}_l(n)\left[ -1\right] , \quad n<0.\\
	\end{cases}
	$$
	The motives $c_qM$  are called the fundamental invariants of $M$ in \cite{huber2006slice}.
	There is an exact triangle in $DM_{-}^{eff}(F)$:
	$$
	\nu_nM\longrightarrow \nu_{<n+1}M\longrightarrow \nu_{<n} M\longrightarrow \nu_nM\left[ 1\right]. 
	$$
	where $\nu_nM=c_n(M)(n)\left[ 2n\right] $ and the functors $\nu_{<n}$ are defined in \cite{huber2006slice}.

	\section{Main results}
	
	\subsection{Motives of classifying space of algebraic tori} \label{subsec 3.1}
	
	We assume that $F$ is a perfect field.
	In this section, we prove that the motive of the classifying space $BT$ of a $F$-torus $T$ is a pure Tate motive, in order to write down the slice spectral sequence for $BT$. As a result, we compute the \'{e}tale motivic cohomology of $BT$ in low degree.
	
	Let $
	1\longrightarrow T\longrightarrow P\longrightarrow Q\longrightarrow 1
	$ be an exact sequence of algebraic tori such that $P$ is quasi-split.
	Since $P$ is quasi-split, we can assume that $P=\prod_{i} R_{K_i/F}(\mathbb{G}_{m,K_i})$, where $K_i / F$ are finite separable field extensions.
	Then the free diagonal action of $\mathbb{G}_{m,K_i}$ on $\mathbb{A}_{K_i}^{r+1}\setminus{\{0\}}$ induces, by Weil restriction, a free action of $P$ on $U=\prod_{i} R_{K_i/F}(\mathbb{A}_{K_i}^{r+1}\setminus{\{0\}})$.
	We write $X$ for $U/T$ and $Y$ for $U/P$.
	
	It is easy to see that $Y$ is geometrically cellular (actually it is cellular since Weil restriction preserves cellular structure).
	Note that $U\to X$ is a classifying $T$-torsor and $U\to Y$ is a classifying $P$-torsor, by \cite[5.3]{garibaldi2003cohomological}. In addition, $X\to Y$ is a $Q$-torsor.
	Let $BT$ (resp. $BP$) be a classifying space of $T$ (resp. $P$).
	Then $BT$ (resp. $BP$) can be approximated by $X$ (resp. $Y$) when $r$ is large enough (see for instance \cite{m2013}, Lemma A.4). In the following, the notation $BT$ will mean the classifying space $X$ constructed here, for $r$ big enough.
	
	In \cite{kahn1999motivic} and \cite{huber2006slice}, the authors provide a natural filtration of the motive of a cellular variety. In particular, since $Y$ is a geometrically cellular variety, there is a spectral sequence $
	E_2^{p,q} = H^{p-q}(F,\mathrm{CH}^q(Y_{sep})\otimes\mathbb Z(n-q))\implies H^{p+q}_\mathrm{\acute{e}t}(Y,\mathbb Z(n))
	$.
	We show below that this spectral sequence also holds for $X$, even though we don't know how to find a classifying torsor $U\to X$ such that $X$ is a geometrically cellular variety and $X$ approximates $BT$.

		\begin{prop}\label{cqBT}
		The complex $c^q(X_{sep}) \in D^b_f(Ab)$ is isomorphic to the following Koszul-like complex $K(X_{sep},q)$:
		\begin{align*}
			\Lambda^{q}\widehat{Q}\longrightarrow \Lambda^{q-1}\widehat{Q}\otimes \mathrm{CH}^{1}(Y_{sep}) \longrightarrow&\Lambda^{q-2}\widehat{Q}\otimes \mathrm{CH}^{2}(Y_{sep}) \longrightarrow\cdots\\
			\longrightarrow&\widehat{Q}\otimes \mathrm{CH}^{q-1}(Y_{sep}) \longrightarrow \mathrm{CH}^q(Y_{sep})
		\end{align*}
	where $\mathrm{CH}^q(Y_{sep})$ is in degree 0, and the maps are induced by interior products and the characteristic map $\theta:\widehat{Q}\to \mathrm{CH}^1(Y_{sep}) $.
	\end{prop}
	\begin{proof}
		The proof is the same as Huber and Khan's proof for split reductive groups \cite[Proposition 9.3]{huber2006slice}.
		Since $Y_{sep}$ is cellular, we have an isomorphism $c^q(Y_{sep})=\mathrm{CH}^q(Y_{sep})\left[ 0\right] $ from \cite[Proposition 4.11]{huber2006slice}.
		Note that $X_{sep}\to Y_{sep}$ is a $Q_{sep}$ torsor.
		Let $\Xi$ be the cocharacter module of $Q$.
		Applying \cite[Proposition 8.10]{huber2006slice} to $X_{sep}\to Y_{sep}$ and to the functor $H^0(c_m): TDM^{eff}_{gm} \to Ab$, there is a spectral sequence 
		$$
		E_1^{p,q} = H^{q}(c_{m-q}(Y_{sep})\otimes \Lambda^p\Xi)\implies H^{p+q}(c_{m}(X_{sep}))
		$$
		for each $m$.
		This spectral sequence is concentrated in the $q = 0$ row and degenerates at $E_2$. Therefore, we get that $c^q(X_{sep})$ is isomorphic to the dual of the complex in $E_1$, whence the result.
		
	\end{proof}

	\begin{lem}\label{koszulcomplex}
		$K(X_{sep}, q)$ is quasi-isomorphic to $S^q(\widehat{T})[0]$.
	\end{lem}
	\begin{proof}
		Note that the sequence $
		1 \to \widehat{Q} \xrightarrow{\alpha} \widehat{P}\xrightarrow{\beta} \widehat{T}\to 1
		$ is exact and that the natural morphism $\theta : \widehat{P} \to \mathrm{CH}^1(Y_{sep})$ induces isomorphisms of Galois modules $\mathrm{CH}^q(Y_{sep}) \cong S^q(\widehat{P})$ for all $p$, by \cite[Lemma 2]{edidin1994characteristic}.
		
		In addition, for all $i$, $j$, the diagram 
        \[
        \xymatrix{
        \Lambda^i \widehat{Q} \otimes S^j(\widehat{P}) \ar[r]^{\varphi_{i,j}} \ar[d]^\sim & \Lambda^{i-1} \widehat{Q} \otimes S^{j+1}(\widehat{P}) \ar[d]^\sim \\
        \Lambda^i \widehat{Q} \otimes \mathrm{CH}^{j}(Y_{sep}) \ar[r] & \Lambda^{i-1} \widehat{Q} \otimes \mathrm{CH}^{j+1}(Y_{sep})
        }
		\]
		is commutative, where the vertical isomorphisms are induced by the aforementionned isomorphisms $S^j(\widehat{P}) \xrightarrow{\sim} \mathrm{CH}^j(Y_{sep})$, the lower horizontal map is the one in the complex $K(X_{sep}, i+j)$, and $\varphi_{i,j}$ is defined by 
		\[\varphi_{i,j}((\chi_1 \wedge \dots \wedge \chi_i) \otimes (\xi_1 \dots \xi_j)) := \sum_{k=1}^i (-1)^{k+1} (\chi_1 \wedge \dots \wedge \widehat{\chi_k} \wedge \dots \chi_i) \otimes (\alpha(\chi_k) \xi_1 \dots \xi_j) \, .\]
	    Therefore, the complex $K(X_{sep}, q)$ is isomorphic to the Koszul complex $K(\alpha, q)$ defined by
	    \[\Lambda^{q}\widehat{Q} \xrightarrow{\varphi_{q,0}} \Lambda^{q-1}\widehat{Q} \otimes \widehat{P} \xrightarrow{\varphi_{q-1, 1}} \Lambda^{q-2}\widehat{Q} \otimes S^{2}(\widehat{P}) \xrightarrow{\varphi_{q-2, 2}} \dots \xrightarrow{\varphi_{2, q-2}} \widehat{Q}\otimes S^{q-1}(\widehat{P}) \xrightarrow{\varphi_{1, q-1}} S^q(\widehat{P}) \, .\]
    	And by \cite{ill71}, Proposition 4.3.1.6, the natural morphism $S^q(\beta) : S^q(\widehat{P}) \to S^q(\widehat{T})$ induces a quasi-isomorphism 
    	\[K(\alpha, q) \xrightarrow{\sim} S^q(\widehat{T}) \, ,\]
		which proves the lemma.
	\end{proof}

	\begin{cor}\label{keyspectralsequence}
		For all $n\ge 0$ there is a spectral sequence
		$$
		E_2^{p,q} = H^{p-q}(F,S^q(\widehat{T})\otimes\mathbb Z(n-q))\implies H^{p+q}_\mathrm{\acute{e}t}(X,\mathbb Z(n)).
		$$
	\end{cor}
	\begin{proof}
		The key point is to show that the motive of $X_{sep}$ is a pure Tate motive in the sense of \cite[Definition 4.9]{huber2006slice}.
		That is, $M(X_{sep})\cong \bigoplus_{p\ge 0}\mathrm{CH}^p(X_{sep})^*\otimes \mathbb{Z}(p)\left[ 2p\right]$ in $DM^{eff}_{gm}(F_{sep})$.
		
		Indeed, there is natural isomorphism (see \cite[Corollary 4.2.5]{voevodsky2000triangulated}):
		$$
		\mathrm{CH}^p(X_{sep})\cong \mathrm{Hom}(M(X_{sep}), \mathbb{Z}(p)\left[ 2p\right]).
		$$
		Since $\mathrm{CH}^p(X_{sep})$ is natural isomorphic to $S^p(\widehat{T})$ \cite[Lemma 2]{edidin1994characteristic}, which is a finitely generated free abelian group, we get a natural map:
		$$\varphi : 
		M(X_{sep})\to \bigoplus_{p\ge 0}\mathrm{CH}^p(X_{sep})^*\otimes \mathbb{Z}(p)\left[ 2p\right].
		$$
		
		Let $M=M(X_{sep}),\ M'=\bigoplus_{p\ge 0}\mathrm{CH}^p(X_{sep})^*\otimes \mathbb{Z}(p)\left[ 2p\right]$ and we now prove that the above morphism $\varphi : M \rightarrow M'$ is an isomorphism.
		First, it is easy to see that $c_nM'=\mathrm{CH}^p(X_{sep})^*$, and Proposition \ref{cqBT} and Lemma \ref{koszulcomplex} imply that the morphism $\varphi$ induces isomorphisms $c_nM \xrightarrow{\sim} c_nM'$, and hence isomorphisms $\nu_nM \xrightarrow{\sim} \nu_nM'$, since by definition $\nu_nM=c_n(M)(n)\left[ 2n\right] $.
		In addition, $\varphi$ induces a natural map of exact triangles
		$$
		\begin{tikzcd}
			\nu_nM\arrow[d,"\simeq"]\arrow[r]&\nu_{<n+1}M\arrow[d]\arrow[r]&\nu_{<n} M\arrow[d]\arrow[r]&\nu_nM\left[ 1\right]\arrow[d,"\simeq"]\\
			\nu_nM'\arrow[r]&\nu_{<n+1}M'\arrow[r]&\nu_{<n} M'\arrow[r]&\nu_nM'\left[ 1\right] \, .
		\end{tikzcd}
		$$
		We deduce by induction that for all $n$, the map $\nu_{<n} M \xrightarrow{\sim} \nu_{<n} M'$ is an isomorphism, hence $\varphi : M \to M'$ itself is an isomorphism since $\nu_{<n} M = M$ and $\nu_{<n}M' = M'$ for $n$ big enough.
		
		This gives us the required spectral sequence by pulling back the filtration to the big \'etale site of $\mathrm{Spec}(F)$.
	\end{proof}
	
	We write $K^M_i$ for the Milnor's $K$-group, $K_i$ for the general $K$-group and $K_i(-)_{ind}$ for the quotient $K_i/K^M_i$.
	
	\begin{cor}\label{keycorollary}
		Let $T$ be a torus over a perfect field $F$ and $X$ as above.
		Then we have:
		$$
		\bar{H}^i_\mathrm{\acute{e}t}(X_{sep},\mathbb{Z}(3))=
		\begin{cases}
			0& i=0,1,2\\
			\widehat{T}\otimes H^1(F_{sep},\mathbb{Z}(2))\simeq \widehat{T}\otimes K_3(F_{sep})_{ind} & i=3\\
			\widehat{T}\otimes K^M_2(F_{sep})& i=4\\
			S^2(\widehat{T})\otimes F^*_{sep}& i=5\\
			S^3(\widehat{T})& i=6\\
		\end{cases}
		$$
		Moreover, there is an exact sequence:
		\begin{align*}
			\left(\widehat{T}\otimes K^M_2(F_{sep})\right)^\Gamma \longrightarrow H^2(F,\widehat{T}\otimes \mathbb{Q/Z}(2)) \longrightarrow \bar{H}^5_\mathrm{\acute{e}t}(X,\mathbb{Z}(3))\longrightarrow& \left(S^2(\widehat{T}) \otimes F_{sep}^*\right)^\Gamma\longrightarrow \\
			H^3(F,\widehat{T}\otimes \mathbb{Q/Z}(2)) \longrightarrow \ker \left(\bar{H}^6_\mathrm{\acute{e}t}(X,\mathbb{Z}(3))\longrightarrow S^3(\widehat{T})^\Gamma \right)\longrightarrow& H^1(F,S^2(\widehat{T})\otimes F^*_{sep}).
		\end{align*}
		This exact sequence holds for $\mathrm{char}\ F=0$. 
		If $\mathrm{char}\ F=p>0$, then the conclusion still holds after changing $\mathbb{Q/Z}(2)$ in the Galois cohomology groups of $F$ to $\mathbb{Q/Z}(2)[1/p]$.
	\end{cor}
	\begin{proof}
		For $\bar{H}^i_\mathrm{\acute{e}t}(X_{sep},\mathbb{Z}(3))$, it is a consequence of Corollary \ref{keyspectralsequence} and \cite[Theorem 1.1]{kahn1996applications} over $F_{sep}$.
		As for the exact sequence of $\bar{H}^i_\mathrm{\acute{e}t}(X,\mathbb{Z}(3))$, from Lemma \ref{k-group} and the Hochschild–Serre spectral sequence (see \cite[Appendix B-IV]{m2013} and  \cite{RS18}, (3.5)):
		$$E_2^{p,q} = H^{p}(F,\bar{H}^q_\mathrm{\acute{e}t}(X_{sep},\mathbb{Z}(3)))\implies \bar{H}^{p+q}_\mathrm{\acute{e}t}(X,\mathbb Z(3)),$$
		we obtain an exact sequence:
		\begin{align*}
			\left(\widehat{T} \otimes K^M_2(F_{sep})\right)^\Gamma \longrightarrow H^2(F,\widehat{T}\otimes K_3(F_{sep})_{ind}) \longrightarrow \bar{H}^5_\mathrm{\acute{e}t}(X,\mathbb{Z}(3))\longrightarrow& \left(S^2(\widehat{T}) \otimes F_{sep}^*\right)^\Gamma\longrightarrow \\
			H^3(F,\widehat{T}\otimes K_3(F_{sep})_{ind}) \longrightarrow \ker \left(\bar{H}^6_\mathrm{\acute{e}t}(X,\mathbb{Z}(3))\longrightarrow S^3(\widehat{T})^\Gamma \right)\longrightarrow& H^1(F,S^2(\widehat{T})\otimes F^*_{sep}).
		\end{align*}
		Because of \cite[Theorem 11.1]{ASM1991} and \cite[VI.1.6]{weibel2013k}, $K_3(F_{sep})_{ind}$ is divisible and its torsion subgroup is $\mathbb{Q/Z}(2)$, therefore $H^i(F,\widehat{T}\otimes K_3(F_{sep})_{ind})\simeq H^i(F,\widehat{T}\otimes \mathbb{Q/Z}(2))$ (see \cite[VI.1.3.1]{weibel2013k} for the case of $\mathrm{char}\ F=p>0$).
		This proves the results.
	\end{proof}

	\begin{lem}\label{k-group}
		$K_i(F)$ is uniquely divisible if $F$ is algebraically closed and $i\ge 2$.
	\end{lem}
	\begin{proof}
		See \cite[Proposition 1.2]{bass1973milnor}.
	\end{proof}

	\begin{cor}\label{keycorollary2}
		Let $T$ be a torus over a perfect field $F$ and $X$ as above.
		Then we have:
		$$
		\bar{H}^i_\mathrm{\acute{e}t}(X_{sep},\mathbb{Z}(4))=
		\begin{cases}
			0 & i=0,1\\
			\widehat{T}\otimes H^{i-2}(F_{sep},\mathbb{Z}(3))& i=2,3,4\\
			S^2(\widehat{T})\otimes K^M_2(F_{sep})& i=6\\
			S^3(\widehat{T})\otimes F^*_{sep}& i=7\\
		\end{cases}
		$$
		and 
		$$
		0\longrightarrow \widehat{T}\otimes K^M_3(F_{sep}) \longrightarrow   \bar{H}^5_\mathrm{\acute{e}t}(X_{sep},\mathbb{Z}(4))\longrightarrow S^2(\widehat{T})\otimes K_3(F_{sep})_{ind} \longrightarrow 0
		$$ 
		is exact.
	\end{cor}
	\begin{proof}
		The proof is similar to the proof of the corollary above.
	\end{proof}

	\subsection{Cohomological invariants of tori}
	
	We prove here the main result describing cohomological invariants of degree $4$ and $5$. 
	We use the notation defined above. 
	Let $A^p(-, K^M_n)$ be the $p$th homology group of complex of cycle modules \cite{rost1996chow}.
	For any $0\leq p\leq n$, from \cite[Lemma 2]{edidin1994characteristic} and Künneth formula \cite[Proposition 3.7]{esnault1998arason}, we have:
	$$
	\begin{tikzcd}\label{cyclemodule}
		A^p(BT, K^M_n)\arrow[r,"\alpha_{p,n}"]& A^p(BT_{sep}, K^M_n)^\Gamma\\
		\mathrm{CH}^p(BT)\otimes K^M_{n-p}(F)\arrow[r]\arrow[u,"\cup"]&\left(S^p(\widehat{T})\otimes K^M_{n-p}(F_{sep})\right)^\Gamma\arrow[u,"\simeq"].
	\end{tikzcd}
	$$
	Let $I$ be the kernel of \'etale cycle map $\mathrm{CH}^3(BT) \to \bar{H}^6_\mathrm{\acute{e}t}(BT,\mathbb{Z}(3))$.
	
	\begin{thm} \label{thminv4}
	Let $T$ be a torus over a perfect field $F$. 
		Then there is a natural commutative diagram:
		\[
		\begin{tikzcd}
			& [-2.5em] \left(\widehat{T} \otimes K^M_2(F_{sep})\right)^\Gamma \arrow[d] &[-2.5em]&[-2em]  \mathrm{Inv}^3(T, \mathbb{Q/Z}(2))_{norm}\otimes F^*\arrow[lldd, "\gamma" above]\arrow[ldd, "\cup" above]&[-1.5em]   \\
			& H^2(F,\widehat{T}\otimes  \mathbb{Q/Z}(2)) \arrow[d]& &0\arrow[d]&\\
			A^2(BT, K_3^M) \arrow[r,hook]&\bar{H}_\mathrm{\acute{e}t}^5(BT, \Z(3))\arrow[d]\arrow[r]&\mathrm{Inv}^4(T, \mathbb{Q/Z}(3))_{norm}\arrow[r]&I \arrow[r]\arrow[d]& 0\\
			&\left(S^2(\widehat{T})\otimes F_{sep}^*\right)^\Gamma\arrow[dr]& &\mathrm{CH}^3(BT)_{tors} \arrow[d]&  \\
			&&H^3(F,\widehat{T}\otimes  \mathbb{Q/Z}(2))\arrow[r]&\ker\left(\bar{H}_\mathrm{\acute{e}t}^6(BT,\Z(3)) \to S^3(\widehat{T})^\Gamma\right) \arrow[r] &H^1(F,S^2(\widehat{T})\otimes F^*_{sep})
		\end{tikzcd}
		\]
		 where the row, the column and the column that turns to a line are exact, and where $\gamma$ is induced by the map $\gamma_0$ in the Introduction.
		This conclusion holds for $\mathrm{char}\ F=0$. 
		If $\mathrm{char}\ F=p>0$, then the conclusion still holds after changing $\mathbb{Q/Z}(2)$ in the Galois cohomology groups of $F$ to $\mathbb{Q/Z}(2)[1/p]$.
	\end{thm}
	\begin{proof}
	    First, the left column of the diagram continuing as the last line in the statement of the theorem is well defined and is exact by Corollary \ref{keycorollary}.
	
		Second, for any smooth variety $X$ over a field $F$, the coniveau spectral sequence gives an exact sequence:
		\begin{align*}
			0\longrightarrow A^2(X, K^M_3)\xrightarrow{f_1} \bar{H}^5_\mathrm{\acute{e}t}(X,\mathbb{Z}(3)) \longrightarrow &\bar{H}_{Zar}^0(X, \mathcal{H}^4(\mathbb{Q/Z}(3)))\\
			\longrightarrow &\mathrm{CH}^3(X) \longrightarrow \bar{H}^6_\mathrm{\acute{e}t}(X,\mathbb{Z}(3)).
		\end{align*}
		Applying this sequence for the classifying $T$-torsor $U^i\to U^i/T$ for every $i> 0$, we obtain an exact sequence
		\begin{align*}
			0\longrightarrow A^2(U^i/T, K^M_3)\longrightarrow \bar{H}^5_\mathrm{\acute{e}t}(U^i/T,\mathbb{Z}(3)) \longrightarrow &\bar{H}_{Zar}^0(U^i/T, \mathcal{H}^4(\mathbb{Q/Z}(3)))\\
			\longrightarrow &\mathrm{CH}^3(U^i/T) \longrightarrow \bar{H}^6_\mathrm{\acute{e}t}(U^i/T,\mathbb{Z}(3)).
		\end{align*}
		Those sequences for all $i$ give an exact sequence of cosimplicial groups (see \cite[A-IV]{m2013}):
		\begin{align*}
			0\longrightarrow A^2(U^\bullet/T, K^M_3)\longrightarrow \bar{H}^5_\mathrm{\acute{e}t}(U^\bullet/T,\mathbb{Z}(3))  \longrightarrow &\bar{H}_{Zar}^0(U^\bullet/T, \mathcal{H}^4(\mathbb{Q/Z}(3)))\\
			\longrightarrow &\mathrm{CH}^3(U^\bullet/T) \longrightarrow \bar{H}^6_\mathrm{\acute{e}t}(U^\bullet/T,\mathbb{Z}(3)).
		\end{align*}
		Since $A^i(-, K^M_j)$ is homotopy invariant \cite[Proposition 8.6]{rost1996chow}, the first and fourth cosimplicial groups in the above sequence are constant \cite[Lemma A.4]{m2013}.
		By Corollary \ref{keycorollary}, each groups $\bar{H}^5_\mathrm{\acute{e}t}(U^i/T,\mathbb{Z}(3))$ only depend on $T$, hence $\bar{H}^5_\mathrm{\acute{e}t}(U^\bullet/T,\mathbb{Z}(3))$ is also constant cosimplicial group.
		Therefore \cite[Lemma A.2]{m2013} and Lemma \ref{invariants} provide an exact sequence:
		\begin{align} \label{middle-exact-seq}
			0 \longrightarrow A^2(BT, K^M_3)\longrightarrow \bar{H}^5_\mathrm{\acute{e}t}(BT,\mathbb{Z}(3)) \longrightarrow &\mathrm{Inv}^4(T, \mathbb{Q/Z}(3))_{norm} \longrightarrow &\mathrm{CH}^3(BT) \longrightarrow \bar{H}^6_\mathrm{\acute{e}t}(U/T,\mathbb{Z}(3)).
		\end{align}
		
		The middle line of the diagram in the Theorem comes from the exact sequence \eqref{middle-exact-seq}, from the vanishing of the invariant group over $F_{sep}$, and from the following commutative diagram:
		$$
		\begin{tikzcd}
			\mathrm{CH}^3(BT)\arrow[d]\arrow[r]&\bar{H}^6_\mathrm{\acute{e}t}(U/T,\mathbb{Z}(3)) \arrow[d]\\
			\mathrm{CH}^3(BT_{sep})^\Gamma\arrow[r,"\cong"]&S^3(\widehat{T})^\Gamma \, .
		\end{tikzcd}
		$$
		Since $\ker(\mathrm{CH}^3(BT)\to \mathrm{CH}^3(BT_{sep}))\cong\mathrm{CH}^3(BT)_{tor}$ by a restriction and corestriction argument, we get the exactness of the whole middle line in the Theorem.
		
		We now prove that the triangle involving the group $\mathrm{Inv}^3(T, \mathbb{Q/Z}(2))_{norm}$ is commutative.
		The following diagram is commutative by definition (the first line is well-defined since $H^4(BT, \Z(2)) \xrightarrow{\sim} \mathrm{CH}^2(BT)$ and $H^5(BT, \Z(2))=0$ for the Zariski topology):
		$$
		\begin{tikzcd}
			\left(\bar{H}^4_\mathrm{\acute{e}t}(X,\mathbb{Z}(2))_{bal}/\mathrm{CH}^2(BT)\right)\otimes H^1_\mathrm{\acute{e}t}(F,\Z(1)) \arrow[d,"\cong"]\arrow[r,"\cup"]&\bar{H}^5_\mathrm{\acute{e}t}(X,\mathbb{Z}(3)) \arrow[d]\\
			\bar{H}^0(X, \mathcal{H}^3(\mathbb{Q/Z}(2)))_{bal}\otimes F^*\arrow[d,"\cong"]\arrow[r,"\cup"]&\bar{H}^0(X, \mathcal{H}^4(\mathbb{Q/Z}(3)))_{bal} \arrow[d,"\cong"]\\
			\mathrm{Inv}^3(T, \mathbb{Q/Z}(2))_{norm}\otimes F^*\arrow[r,"\cup"]&\mathrm{Inv}^4(T, \mathbb{Q/Z}(3))_{norm}.
		\end{tikzcd}
		$$
		Therefore, we get the required commutativity by inverting the isomorphisms.
	\end{proof}
	
	
Similarly,	for degree $5$ cohomological invariants, we can also deduce from the coniveau spectral sequence an exact sequence:
	\begin{align*}
		0\longrightarrow A^2(X, K^M_4)\xrightarrow{f_2} \bar{H}^6_\mathrm{\acute{e}t}(X,\mathbb{Z}(4)) \longrightarrow &\bar{H}_{Zar}^0(X, \mathcal{H}^5(\mathbb{Q/Z}(4)))\\
		\longrightarrow & A^3(X, K^M_4) \longrightarrow \bar{H}^7_\mathrm{\acute{e}t}(X,\mathbb{Z}(4)).
	\end{align*}
	Combining this exact sequence and the argument in the proof above,  we obtain an exact sequence involving degree $5$ cohomological invariants:
	\[
		0 \longrightarrow A^2(BT, K^M_4) \longrightarrow \bar{H}^6_\mathrm{\acute{e}t}(BT,\mathbb{Z}(4)) \longrightarrow \mathrm{Inv}^5(T, \mathbb{Q/Z}(4))_{norm}
		\longrightarrow A^3(BT, K^M_4)  \longrightarrow \bar{H}^7_\mathrm{\acute{e}t}(BT,\mathbb{Z}(4)) \, .
	\]

    The Hochschild–Serre spectral sequence (see \cite{RS18}, (3.5)):
		$$E_2^{p,q} = H^{p}(F,\bar{H}^q_\mathrm{\acute{e}t}(BT_{sep},\mathbb{Z}(4)))\implies \bar{H}^{p+q}_\mathrm{\acute{e}t}(BT,\mathbb Z(4)) \, ,$$
    and Corollary \ref{keycorollary2} give rise to the Theorem below, taking into account that the groups $H^i(F_{sep},\Z(3))$ are uniquely divisible for $i=0$ and $i=2$ (see for instance \cite{G17}, Theorem 1.1):
    
	
	\begin{thm}\label{thminv5}
		Let $T$ be a torus over a perfect field $F$. Then there is an exact commutative diagram:
		\[
		\begin{tikzcd}
		& [-6.5em] H^3(F, \widehat{T} \otimes H^1(F_{sep},\Z(3))) \arrow[d] & [-4.5em] & [-2.5em] & [-6.5em] H^4(F,  \widehat{T} \otimes H^1(F_{sep},\Z(3))) \arrow[d] & [-6.5em]\\
		A^2(BT, K^M_4) \arrow[r,hook] & \bar{H}^6_\mathrm{\acute{e}t}(BT,\mathbb{Z}(4)) \arrow[r] \ar[d] & \mathrm{Inv}^5(T, \mathbb{Q/Z}(4))_{norm} \arrow[r] & A^3(BT,K^M_4) \arrow[r] & \bar{H}^7_\mathrm{\acute{e}t}(BT,\mathbb{Z}(4)) \arrow[d] & \\
		H^1(F, S^2(\widehat{T}) \otimes \Q/\Z(2))' \arrow[r,hook] & Q \arrow[r] \arrow[d]& (S^2(\widehat{T})\otimes K^M_2(F_{sep}))^\Gamma & H^2(F, S^2(\widehat{T}) \otimes \Q/\Z(2))' \arrow[r,hook] & R \arrow[r] \arrow[d] & \left(S^3(\widehat{T})\otimes F_{sep}^*\right)^\Gamma \\
		&0& & &0& \\
		\end{tikzcd}
		\]
		where $H^1(F, S^2(\widehat{T}) \otimes \Q/\Z(2))' := \ker\left(H^1(F, S^2(\widehat{T}) \otimes \Q/\Z(2)) \to H^4(F, \widehat{T} \otimes H^1(F_{sep},\Z(3)))\right)$ and
		$H^2(F, S^2(\widehat{T}) \otimes \Q/\Z(2))' := \ker\left(H^2(F, S^2(\widehat{T}) \otimes \Q/\Z(2))  \to H^5(F, \widehat{T} \otimes H^1(F_{sep},\Z(3)))\right)/ (S^2(\widehat{T})\otimes K^M_2(F_{sep}))^\Gamma$.
	\end{thm}

	\subsection{Unramified cohomology}
	
	Given a field extension $L/F$, the $i$-th unramified cohomology group of $L/F$ is defined as the group
	$$
	H_{nr}^i(L/F,\Q/\Z(j))):=\bigcap_{A\in P(L)}\mathrm{im}\ (H^i(A,\Q/\Z(j))\longrightarrow H^{i}(L,\Q/\Z(j))) \, ,
	$$
	where $P(L)$ is the set of all rank one discrete valuation rings which contain $F$ and have quotient field $L$.
	They can also be defined by the intersection of the kernel of the residue maps $\partial_A$, for $A\in P(L)$ when $\mathrm{char}(F)=0$ \cite{ct1995}.
	If $X$ is a smooth integral variety over $F$, the $i$-th unramified cohomology group of $X$ is defined as $H_{nr}^i(F(X)/F,\Q/\Z(j))$.

	\begin{cor}\label{unramified}
		Let $S$ be a torus over a perfect field $F$, and let $$1\longrightarrow T\longrightarrow P\longrightarrow S\longrightarrow 1$$ be a flasque resolution of $S$.
		Then the following commutative diagrams are exact:
		\[
		\begin{tikzcd}
			& [-2.5em] \left(\widehat{T} \otimes K^M_2(F_{sep})\right)^\Gamma \arrow[d] &[-2.5em]&[-2em]  \bar{H}^3_{nr}(F(S), \mathbb{Q/Z}(2))\otimes F^*\arrow[lldd, "\gamma" above]\arrow[ldd, "\cup" above]&[-1.5em]   \\
			& H^2(F,\widehat{T}\otimes  \mathbb{Q/Z}(2)) \arrow[d]& &0\arrow[d]&\\
			A^2(BT, K_3^M) \arrow[r,hook]&\bar{H}_\mathrm{\acute{e}t}^5(BT, \Z(3))\arrow[d]\arrow[r]&\bar{H}^4_{nr}(F(S), \mathbb{Q/Z}(3))\arrow[r]&I \arrow[r]\arrow[d]& 0\\
			&\left(S^2(\widehat{T})\otimes F_{sep}^*\right)^\Gamma\arrow[dr]& &\mathrm{CH}^3(BT)_{tors} \arrow[d]&  \\
			&&H^3(F,\widehat{T}\otimes  \mathbb{Q/Z}(2))\arrow[r]&\ker\left(\bar{H}_\mathrm{\acute{e}t}^6(BT,\Z(3)) \to S^3(\widehat{T})^\Gamma\right) \arrow[r] &H^1(F,S^2(\widehat{T})\otimes F^*_{sep})
		\end{tikzcd}
		\]
		and
        \[
        \begin{tikzcd}
        	& [-6.5em] H^3(F, \widehat{T} \otimes H^1(F_{sep},\Z(3))) \arrow[d] & [-4.5em] & [-2.5em] & [-6.5em] H^4(F,  \widehat{T} \otimes H^1(F_{sep},\Z(3))) \arrow[d] & [-6.5em]\\
        	A^2(BT, K^M_4) \arrow[r,hook] & \bar{H}^6_\mathrm{\acute{e}t}(BT,\mathbb{Z}(4)) \arrow[r] \ar[d] & \bar{H}^5_{nr}(F(S), \mathbb{Q/Z}(4)) \arrow[r] & A^3(BT,K^M_4) \arrow[r] & \bar{H}^7_\mathrm{\acute{e}t}(BT,\mathbb{Z}(4)) \arrow[d] & \\
        	H^1(F, S^2(\widehat{T}) \otimes \Q/\Z(2))' \arrow[r,hook] & Q \arrow[r] \arrow[d]& (S^2(\widehat{T})\otimes K^M_2(F_{sep}))^\Gamma & H^2(F, S^2(\widehat{T}) \otimes \Q/\Z(2))' \arrow[r,hook] & R \arrow[r] \arrow[d] & \left(S^3(\widehat{T})\otimes F_{sep}^*\right)^\Gamma \\
        	&0& & &0& \\
        \end{tikzcd}
        \]
		where $H^1(F, S^2\widehat{T} \otimes \Q/\Z(2))' := \ker\left(H^1(F, S^2\widehat{T} \otimes \Q/\Z(2)) \to H^4(F, \widehat{T} \otimes H^1(F_{sep},\Z(3)))\right)$ and
		$H^2(F, S^2\widehat{T} \otimes \Q/\Z(2))' := \ker\left(H^2(F, S^2\widehat{T} \otimes \Q/\Z(2)) / (S^2(\widehat{T})\otimes K^M_2(F_{sep}))^\Gamma \to H^5(F, \widehat{T} \otimes H^1(F_{sep},\Z(3)))\right)$.
		
		This conclusion holds for $\mathrm{char}\ F=0$. 
		If $\mathrm{char}\ F=p>0$, then the conclusion still holds after changing $\mathbb{Q/Z}(2)$ in the Galois cohomology groups of $F$ to $\mathbb{Q/Z}(2)[1/p]$.
	\end{cor}
	\begin{proof}
		It is a consequence of Theorem 5.7 in \cite{m2013} and of Theorems \ref{thminv4} and \ref{thminv5}.
	\end{proof}
	
	\section{Examples}
	
	In this section, we apply the main results to compute explicitely invariant groups for some families of tori. In particular, we provide an example of torus with a non-trivial invariant of degree $4$, and we compute the invariant groups in the case of norm tori.
	
	\begin{exa*}
		Let us construct a torus $T$ with a non-trivial degree $4$ invariant that does not come from a lower degree invariant.
	\end{exa*}

		In \cite{sala2022chow}, Sala constructs a torus $T$ over a field $F$ such that $\mathrm{CH}^3(BT)_{tors}$ is nontrivial, as follows: consider the following exact sequence of $Q_8$-modules
		$$
		0\longrightarrow \widehat{P}\longrightarrow \widehat{Q}\longrightarrow \widehat{T} := \widehat{Q}/\widehat{P} \longrightarrow 0 \, ,
		$$
		where $\widehat{Q}$ is $\mathbb{Z}\left[ Q_8\right] $ and $\widehat{P}$ is $\mathbb{Z}\left[ Q_8/\left\lbrace 1,-1\right\rbrace \right] $.
		Here $Q_8=\left\lbrace i,j,k|i^2=j^2=k^2=ijk=-1 \right\rbrace $ is the quaternion group of order $8$.
		Let $e,e',x,x',y,y',z,z' $ denote the elements associated respectively to $1,-1,i,-i,j,-j,k,-k$ inside the character group $\widehat{Q}$.
		Then $(e+e',x+x',y+ y',z+ z')$ is a $\Z$-basis of the sublattice $\widehat{P}$. In addition, the classes of $e,x,y,z$ in $\widehat{T}$ are a $\Z$-basis of $\widehat{T}$. Then we have a decomposition as a $Q_8$-module:
		\[
		S^2(\widehat{T}) = M \oplus P_x \oplus P_y \oplus P_z \, ,
		\]
		where the submodules on the right hand side are defined by $M := \mathbb{Z} ee \oplus \Z xx \oplus \Z y y \oplus \Z z z$, $P_x := \mathbb{Z}(ex+yz) \oplus \mathbb{Z} (ex-yz)$, $P_y := \mathbb{Z}(ey+xz) \oplus \mathbb{Z} (ey-xz)$ and $P_z := \mathbb{Z}(ez+xz) \oplus \mathbb{Z} (ez-xz)$.
		It is easy to see that $M$ is a rank $4$ permutation $Q_8$-module which is isomorphic to $\mathbb{Z} \left[Q_8/\{1,-1\}\right]$.
		
		Let $N_x := \Z(ex-yz) \subset P_x$ be the rank 1 submodule generated by $ex-yz$, and let $\bar{N}_x := P_x / N_x$. Then we have an exact sequence of $Q_8$-modules
		\[0 \to N_x \to P_x \to \bar N_x \to 0\]
		and $N_x$ (resp. $\bar N_x$) is isomorphic to $\Z$ with the non-trivial action of $Q_8 / \langle j \rangle$ (resp. of $Q_8 / \langle k \rangle$).
		
		Let now $L/F$ be a Galois extension with Galois group $Q_8$. Let $K$ be the subfield fixed by $\{1,-1\}$ and $K_1$, $K_2$, $K_3$ be the subfields of $K$ with Galois groups $\{1,\bar{i}\}$, $\{1,\bar{j}\}$, $\{1,\bar{k}\}$ over $F$ respectively.
		
		By Shapiro Lemma and Hilbert 90, we have $H^1(F,M \otimes L^*) \cong H^1(Q_8, M \otimes L^*) \cong H^1(L/K, L^*)=0$, and exact sequences
		\[0 \to H^1(F, N_x \otimes L^*) \to \mathrm{Br}(F) \to \mathrm{Br}(K_2)\]
		and 
		\[0 \to H^1(F, \bar N_x \otimes L^*) \to \mathrm{Br}(F) \to \mathrm{Br}(K_3) \, .\]
		Hence, the exact sequence
		\[0 \to H^1(F, N_x \otimes L^*) \to H^1(F, P_x \otimes L^*) \to H^1(F, \bar N_x \otimes L^*) \]
		gives rise to an exact sequence 
		\[0 \to \mathrm{Br}(K_2/F) \to H^1(F, P_x \otimes L^*) \to \mathrm{Br}(K_3/F) \, . \]
		
		Using a similar argument for $P_y$ and $P_z$, we get an exact sequence:
		\[0 \to \mathrm{Br(K_1/F)} \oplus \mathrm{Br(K_2/F)} \oplus \mathrm{Br(K_3/F)} \to 
			H^1(F,S^2(\widehat{T})\otimes F^*_{sep}) \to  \mathrm{Br(K_1/F)} \oplus \mathrm{Br(K_2/F)} \oplus \mathrm{Br(K_3/F)} \, .\]
	
		If we restrict to $F$ being the maximal abelian extension of an algebraic extension of $\mathbb{Q}$, then these relative Brauer groups are trivial. 
		Hence $H^1(F,S^2(\widehat{T})\otimes F^*_{sep}) = 0$, since $\mathbb{Q}^{ab}$ has cohomological dimension $1$, therefore the non-trivial element in $\mathrm{CH}^3(BT)_{tor}$ given by Sala \cite[Theorem 5.8]{sala2022chow} defines a non-trivial degree $4$ invariant by Theorem \ref{thminv4}, and this invariant does not come from degree $3$ cohomological invariants by cup with $F^*$.
		Note that $Q_8$ is indeed a Galois group of $F$ by \cite{iwasawa1953solvable}.
	
	In particular, we proved the following:
	\begin{prop}
	There exists an explicit field $F$ of characteristic zero and an explicit $4$-dimensional $F$-torus $T$ with a non-trivial degree $4$ invariant that does not come from a degree $3$ invariant by cup-product with $F^*$.
	\end{prop}
		
	\begin{exa*} Let us consider now a torus $T$ that fits into an exact sequence:
		$$
		1\longrightarrow T\longrightarrow P\longrightarrow \mathbb{G}_m^n\longrightarrow 1
		$$
		where $P$ is a quasi-trivial torus. In particular, norm one tori fit into such exact sequences with $n=1$.
	\end{exa*}
	
		
		
	From the construction at the beginning of section \ref{subsec 3.1}, the natural map $BT\to BP$ is a $\mathbb{G}_m^n$-torsor and $BP$ is approximated by cellular varieties.
		
		\begin{prop}\label{A-group}
			Let $T$ be a torus as above over an infinite field $F$, then $A^i(BT,K^M_j)\simeq S^i(\widehat{T})^\Gamma\otimes K^M_{j-i}(F)$.
			In particular, if $i=j$, $\mathrm{CH}^i(BT)\simeq S^i(\widehat{T})^\Gamma$.
		\end{prop}
		\begin{proof}
			Let $X$ and $Y$ be as in the beginning of section \ref{subsec 3.1}, approximating $BT$ and $BP$ respectively. Then the map $X\to Y$ is a $\mathbb{G}_m^n$-torsor.
			We may assume that $Y$ is a cellular variety.
			Let $d$ be the dimension of $X$.
			From the arguments of \cite[3.10-3.12]{esnault1998arason} and \cite[Proposition 4.4]{sala2022chow}, we obtain a spectral sequence:
			$$
			E^1_{p,q}=\Lambda^p\mathbb Z^n \otimes S^{d-p-q}(\widehat{P})^\Gamma\otimes K^M_{j-q}(F)\implies A^{d-p-q}(X,K^M_{j}).
			$$
			Therefore, the second page is $E^2_{0,q}=S^{d-q}(\widehat{T})^\Gamma\otimes K^M_{j-q}(F)$ and $E^2_{p,q}=0$ for $p\neq 0$ because of Proposition \ref{cqBT} and Lemma \ref{koszulcomplex}.
			This proves $A^i(BT,K^M_j)\simeq S^i(\widehat{T})^\Gamma\otimes K^M_{j-i}(F)$.
		\end{proof}

		Theorefore, the group of degree $3$ normalized cohomological invariants of $T$ is isomorphic to \[\mathrm{Inv}^3(T,\mathbb{Q/Z}(2))_{norm} \simeq H^1(F,T^\circ)\] by Theorem \ref{invariants3}.  
		
		If in particular $T = R^{(1)}_{L/F}(\mathbb{G}_m)$, $P = R_{L/F}(\mathbb{G}_m)$ and $n=1$, we get 
		\[\mathrm{Inv}^3(T,\mathbb{Q/Z}(2))_{norm} \simeq H^1(F,T^\circ) \simeq \mathrm{Br}(L/F) \, .\]
		
		From Proposition \ref{A-group}, $A^2(BT, K^M_3)\simeq S^2(\widehat{T})^\Gamma \otimes F^*$ and $A^3(BT, K^M_3)\simeq S^3(\widehat{T})^\Gamma$.
		Therefore, considering the commutative diagram:
		$$
		\begin{tikzcd}
			S^2(\widehat{T})^\Gamma \otimes F^* \simeq A^2(BT, K_3^M) \arrow[r,hook]\arrow[dr,hook]&\bar{H}_\mathrm{\acute{e}t}^5(BT, \Z(3))\arrow[d]\\
			&\left(S^2(\widehat{T})\otimes F_{sep}^*\right)^\Gamma \, ,
		\end{tikzcd}
		$$
		Theorem \ref{thminv4} implies that the following exact sequence computes the group of degree $4$ invariants:
		\begin{align*}
			0\to H^2(F,\widehat{T}\otimes  \mathbb{Q/Z}(2))/\left(\widehat{T} \otimes K^M_2(F_{sep})\right)^\Gamma &\to 	\mathrm{Inv}^4(T, \mathbb{Q/Z}(3))_{norm} \\
		 	&	\to\left(S^2(\widehat{T})\otimes  F_{sep}^*\right)^\Gamma/\left( S^2(\widehat{T})^\Gamma\otimes  F_{sep}^*\right) \to H^3(F,\widehat{T}\otimes \mathbb{Q/Z}(2)) \, .
		\end{align*}
		
		For degree $5$ invariants, assume for simplicity that, from now on, $F$ is of cohomological dimension $\leq1$. Then Theorem \ref{thminv4} can be simplified as follows:
		\[
		\begin{tikzcd}
			&  H^1(F, S^2(\widehat{T}) \otimes \Q/\Z(2)) \arrow[d, hook] & \\
			A^2(BT, K^M_4) \arrow[r,hook] & \bar{H}^6_\mathrm{\acute{e}t}(BT,\mathbb{Z}(4)) \arrow[r, two heads] \arrow[d, two heads] & \mathrm{Inv}^5(T, \mathbb{Q/Z}(4))_{norm} \\
			&  \left(S^2(\widehat{T})\otimes K^M_2(F_{sep})\right)^\Gamma &   \\
		\end{tikzcd}
		\]
		Therefore we obtain an exact sequence:
		\begin{align*}
		    0\to \ker\left(S^2(\widehat{T})^\Gamma \otimes K_2^M(F) \to S^2(\widehat{T}) \otimes K_2^M(F_{sep}) \right) & \to  H^1(F,S^2(\widehat{T})\otimes \mathbb{Q/Z}(2) ) \to \mathrm{Inv}^5(T, \mathbb{Q/Z}(4))_{norm} \\
		    & \to \left(S^2(\widehat{T})\otimes K^M_2(F_{sep})\right)^\Gamma/\left( S^2(\widehat{T})^\Gamma\otimes K^M_2(F)\right)\to 0 \, .
		\end{align*}
		
		Let $K/F$ be a splitting field of $T$ and $G$ be its Galois group.
		The assumption that $F$ has dimension $1$ provides an exact sequence for $H^1(F,S^2(\widehat{T})\otimes \mathbb{Q/Z}(2) )$ by Hochschild-Serre spectral sequence:
		\begin{align*}
		0\to  H^1(G,S^2(\widehat{T})\otimes \mathbb{Q/Z}(2)^{\Gamma_K} ) \to H^1(F,S^2(\widehat{T})\otimes \mathbb{Q/Z}(2) ) &\to H^1(K,S^2(\widehat{T})\otimes \mathbb{Q/Z}(2) )^G\\
		&\to H^2(G,S^2(\widehat{T})\otimes \mathbb{Q/Z}(2)^{\Gamma_K} )  \to 0.
		\end{align*}
		
		The proof of Lemma \ref{koszulcomplex} provides an exact sequence:
		$$
		0\to \Lambda^2\mathbb{Z}^n\to \mathbb{Z}^n\otimes \widehat{P}\to S^2(\widehat{P})\to S^2(\widehat{T})\to 0.
		$$
		Note that the composition $S^2(\widehat{P})\to \widehat{P}\otimes \widehat{P}\to S^2(\widehat{P})$ is multiplication by $2$, where the first homomorphism maps $\ x\cdot y$ to $x\otimes y+ y\otimes x$ and the second is the natural quotient.
		If we assume that $\widehat{P}$ is a direct sum of several copies of $\mathbb{Z}[G]$ (for instance, if $T$ is norm one torus) and $\vert G \vert$ is odd (or consider the $p$-part of each groups, $p\neq 2$), 
		then $H^i(G,S^2(\widehat{P})\otimes \mathbb{Q/Z}(2)^{\Gamma_K}  )$ is trivial for $i\ge 1$ because of the Shapiro's lemma and we get $ H^i(G,S^2(\widehat{T})\otimes \mathbb{Q/Z}(2)^{\Gamma_K}  )\simeq H^{i+2}(G,\mathbb{Q/Z}(2)^{\Gamma_K}  )^{\oplus {n(n-1)\over 2}}$ for $i\ge 1$.
		Therefore, 
		in this case, we obtain an exact sequence describing $H^1(F,S^2(\widehat{T})\otimes \mathbb{Q/Z}(2) )$:
		\begin{align*}
			0\to  H^3(G,\mathbb{Q/Z}(2)^{\Gamma_K} )^{\oplus {n(n-1)\over2}} \to H^1(F,S^2(\widehat{T})\otimes \mathbb{Q/Z}(2) ) &\to \left(S^2(\widehat{T})\otimes H^1(K, \mathbb{Q/Z}(2) )\right)^G\\
			&\to H^4(G,\mathbb{Q/Z}(2)^{\Gamma_K} )^{\oplus {n(n-1)\over2}}  \to 0.
		\end{align*}

	\bigskip
	{\bf Acknowledgements. }
	The second author acknowledges financial support from the China Scholarship Council (CSC) during his visit to IMJ-PRG, Sorbonne Université. He also thanks his advisor, Prof. Dasheng Wei, for his support.
	
	\bibliography{ref}

\end{document}